\documentclass[12pt]{amsart}

\usepackage{amssymb}
\usepackage{latexsym}
\usepackage{verbatim,euscript}
\usepackage{fullpage,color}

\usepackage[matrix,arrow,curve]{xy}

\input {cyracc.def}
\numberwithin{equation}{section}

\usepackage[colorlinks=true,linkcolor=blue,urlcolor=my_color,citecolor=magenta]{hyperref}

\definecolor{my_color}{rgb}{0,0.5,0.5}
\definecolor{MIXT}{rgb}{0.4,0.3,0.6}

\newtheorem{thm}{Theorem}[section]

\newtheorem{lm}[thm]{Lemma}
\newtheorem{cor}[thm]{Corollary}
\newtheorem{prop}[thm]{Proposition}

\theoremstyle{remark}
\newtheorem{rmk}[thm]{Remark}

\theoremstyle{definition}
\newtheorem{ex}[thm]{Example}
\newtheorem*{ex-bn}{Example}
\newtheorem*{rema}{Remark}
\newtheorem{df}{Definition}

\newenvironment{proof*}
{\noindent {\sl Proof.}\quad }{\hfill
$\square$}

\newcommand {\ah}{{\mathfrak a}}
\newcommand {\be}{{\mathfrak b}}
\newcommand {\ce}{{\mathfrak c}}

\newcommand {\g}{{\mathfrak g}}

\newcommand {\el}{{\mathfrak l}}

\newcommand {\n}{{\mathfrak n}}

\newcommand {\p}{{\mathfrak p}}
\newcommand {\q}{{\mathfrak q}}

\newcommand {\te}{{\mathfrak t}}
\newcommand {\ut}{{\mathfrak u}}
\newcommand {\z}{{\mathfrak z}}

\newcommand {\slno}{\mathfrak{sl}_{l+1}}

\newcommand {\spv}{\mathfrak{sp}(V)}
\newcommand {\spn}{\mathfrak{sp}_{2l}}

\newcommand {\sov}{\mathfrak{so}(V)}

\newcommand {\son}{\mathfrak{so}_{n}}

\newcommand {\eus}{\EuScript}

\newcommand {\ap}{\alpha}

\newcommand {\lb}{\lambda}

\newcommand {\cF}{{\mathcal F}}
\newcommand {\cH}{{\mathcal H}}

\newcommand {\BZ}{{\mathbb Z}}
\newcommand {\BN}{{\mathbb N}}

\newcommand {\md}{/\!\!/}

\newcommand {\ads}{{\mathrm{ad}^*}}
\newcommand {\Ad}{{\mathrm{Ad}}}

\newcommand {\codim}{{\mathrm{codim\,}}}

\newcommand {\ind}{{\mathsf{ind\,}}}
\newcommand {\Lie}{{\mathrm{Lie\,}}}

\newcommand {\Ima}{{\mathrm{Im}}}

\newcommand {\rk}{{\mathrm{rk\,}}}
\newcommand {\spe}{{\mathsf{Spec\,}}}

\newcommand {\trdeg}{{\mathrm{tr.deg\,}}}

\newcommand {\tri}{\mathfrak{sl}_2}
\newcommand {\GR}[2]{{\textrm{{\bf #1}}}_{#2}}

\newcommand {\ov}{\overline}
\newcommand {\un}{\underline}

\newcommand {\beq}{\begin{equation}}
\newcommand {\eeq}{\end{equation}}

\renewcommand{\le}{\leqslant}
\renewcommand{\ge}{\geqslant}

\newcommand {\bbk}{\Bbbk}

\newcommand{\gt}{\mathfrak}

\begin{document}
\setlength{\parskip}{2pt plus 4pt minus 0pt}
\hfill {\scriptsize December 31, 2012}
\vskip1.5ex

\title[Parabolic contractions of semisimple Lie algebras]
{Parabolic contractions of semisimple Lie algebras \\and their invariants}
\author[D.\,Panyushev]{Dmitri I.~Panyushev}
\address[D.P.]{Institute for Information Transmission Problems of the Russian Academy of Sciences, 
 B. Karetnyi per. 19, Moscow 127994, Russia}
\email{panyushev@iitp.ru}
\author[O.\,Yakimova]{Oksana S.~Yakimova}
\address[O.Y.]{Mathematisches Institut, Friedrich-Schiller-Universit\"at Jena,
Deutschland}
\email{oksana.yakimova@uni-jena.de}
\maketitle

\section*{Introduction}
\noindent
The ground field $\bbk$ is algebraically closed and $\mathsf{char}\,\bbk=0$.
Let $G$ be a connected semisimple algebraic group of rank $l$, with  Lie algebra $\g$. 
Motivated by some problems in Representation Theory \cite{ffp,ffp2}, E.\,Feigin introduced recently a 
very interesting contraction of $\g$~\cite{feigin1}. 
This contraction is the semi-direct product $\tilde\q=\be\ltimes (\g/\be)^a$, where 
$\be$ is a Borel subalgebra of $\g$ and the $\be$-module $\g/\be$ is regarded as 
an abelian ideal in $\tilde\q$. 
Using this contraction, Feigin also defined certain degenerations of the usual flag variety of $G$.
This leads to numerous problems of algebraic-geometric and combinatorial nature, see 
\cite{ifr12,feigin2,feigin3}. 
Our intention is to look at $\tilde\q$ from the invariant-theoretic point of view.
In \cite{alafe1}, we proved that the ring of invariants for the adjoint or coadjoint representation 
of $\tilde\q$ is always polynomial and  that the enveloping algebra, 
$\eus U(\tilde\q)$, is a free module over its centre. 
In this paper, we generalise Feigin's construction by replacing $\be$ with an arbitrary parabolic 
subalgebra $\p\subset\g$. The resulting Lie algebras are said to be {\it parabolic contractions\/} of 
$\g$. For arbitrary parabolic contractions, 
the description of the invariants of the adjoint representation is easy and remains basically the 
same as for $\p=\be$, while the case of the coadjoint 
representation requires new techniques. 

Let $P$ be a parabolic subgroup of $G$ with $\Lie P=\p$ and $\n$  the nilpotent radical of $\p$.
Fix a Levi subgroup $L\subset P$ and
a vector space decomposition
$\g=\n\oplus\el\oplus\n_-$, where $\el=\Lie L$
and $\n_-$ is the nilpotent radical of an opposite parabolic subalgebra $\p_-=\el\oplus\n_-$.
Using the vector space isomorphism $\g/\p\simeq\n_-$, we always regard $\n_-$ as a $P$-module. 
If $p\in\p$, $\eta \in \n_-$, and $\mathsf{pr}_-: \g\to \n_-$ is the 
projection with kernel $\p$, then the corresponding representation of $\p$ is given by 
$(p,\eta) \mapsto p\circ \eta:= \mathsf{pr}_-([p,\eta])$. 
A {\it parabolic contraction\/} of $\g$ is the  semi-direct product
$\q=\p\ltimes (\g/\p)^a=\p\ltimes \n_-^a$, where
the superscript `$a$' means that the $\p$-module $\n_-$ is regarded as an abelian ideal in 
$\q$. We  identify the vector spaces $\g$ and $\q$ using the decomposition
$\g=\p\oplus\n_-$.
For $(p,\eta), (p',\eta') \in \q$, the Lie bracket in $\q$ is defined by
\beq  \label{eq:skobka}
   [(p,\eta), (p',\eta')]=([p,p'], p\circ\eta'- p'\circ\eta) .
\eeq
Set $N_-^a=\exp(\n_-^a)$ and  $Q=P\ltimes N_-^a$. Then $Q$ is a connected algebraic 
group with $\Lie Q=\q$ and $N_-^a$ is an abelian normal unipotent subgroup
of $Q$. 
The exponential map $\exp: \n_-^a \to N_-^a$ is an isomorphism of varieties, and 
elements of $Q$ can be written as products $s{\cdot}\exp(\eta)$ with $s\in P$ and $\eta\in\n_-$.
If  $(s,\eta)\mapsto s{\centerdot}\eta$ is the representation of $P$ in $\n_-$, then
the adjoint representation of $Q$ is given by
\beq   \label{eq:adj-action}
   \Ad_Q(s{\cdot}\exp(\eta))(p,\eta')=(\Ad(s)p, s{\centerdot} (\eta' - p\circ\eta)) .
\eeq
In this article, we consider  polynomial invariants of the adjoint and coadjoint  representations 
of $Q$. In the adjoint case the answer is uniform, nice, and easy. We prove that
\[
   \bbk[\q]^Q\simeq \bbk[\p]^P \simeq \bbk[\el]^L 
\]
and the quotient morphism $\pi_\q:  \q\to \q\md Q$ is equidimensional.
In particular, $\bbk[\q]^Q$ is a graded polynomial algebra with $l$ generators (see Section~\ref{sect2}).
However, the degrees of basic invariants in $ \bbk[\q]^Q$ and $\bbk[\g]^G$ are not the same (unless
$L=P=G$).

In the coadjoint case the situation is more complicated and interesting. Our main observation is that
the structure of $\bbk[\q^*]^Q$ is closely related to some properties of the centraliser
$\g_e\subset\g$, where $e\in \n$ is a Richardson element associated with $\p$. 
It is known that $\g_e=\p_e$ and therefore the groups $P_e\subset G_e$ have the same identity component.
Let $\eus S(\g_e)$ be the symmetric algebra of $\g_e$ and  $\eus S(\g_e)^{G_e}$ the subalgebra
of symmetric invariants. Using an $\tri$-triple containing $e$ and the 
(homogeneous) basic invariants $\cF_1,\dots,\cF_l$
in $\eus S(\g)^G$, one can construct certain polynomials 
$^{e\!}\cF_1,\ldots,\,^{e\!}\cF_l\in\eus S(\g_e)^{G_e}$ (see \cite{ppy} and Section~\ref{sect3}).
We prove that if $^{e\!}\cF_1,\ldots,\,^{e\!}\cF_l$ are algebraically independent and generate the 
algebra $\eus S(\g_e)^{P_e}$ (hence $\eus S(\g_e)^{P_e}=\eus S(\g_e)^{G_e}$), then 
$\bbk[\q^*]^Q=\eus S(\q)^Q$ is a polynomial algebra whose free generators
$\cF_1^\bullet,\dots, \cF_l^\bullet$ are obtained from $\cF_1,\dots,\cF_l$ via a standard contraction
procedure (see Theorems~~\ref{thm:main3} and \ref{thm:coadj-inonu}). 
In this situation, 
$\cF_1^\bullet,\dots, \cF_l^\bullet$ have also
the following Kostant-like property: the differentials ${\textsl d}_\xi \cF_1,\dots,{\textsl d}_\xi \cF_l$ 
are linearly independent ($\xi\in\q^*$)
if and only if the orbit $Q{\cdot}\xi\subset \q^*$ has the maximal dimension.
This relies on the theory developed by the second author in \cite{kys-1param}.
Since $\deg \cF_i^\bullet=\deg \cF_i$, we see that, 
unlike the case of the adjoint representation of $\q$, the algebras $\eus S(\g)^G$ and
$\eus S(\q)^Q$ here have the same degrees of basic invariants.

A lot of information on the algebras $\eus S(\g_e)^{G_e}$ and $\eus S(\g_e)^{\g_e}$ is 
obtained in \cite{ppy}, and translating 
some of those results in the setting of parabolic contractions yields applications of 
Theorem~\ref{thm:main3}.
For $\g$ of type $\GR{A}{l}$ or $\GR{C}{l}$, \un{all} nilpotent elements $e$ satisfy the 
above condition on $^{e\!}\cF_1,\ldots,\,^{e\!}\cF_l$ (see \cite[Section\,4]{ppy}), which implies that 
$\eus S(\q)^Q$ is a polynomial algebra for \un{all} parabolic contractions of $\g=\slno$ or $\spn$.
For $\g$ of type $\GR{B}{l}$, the same result is obtained for a special class of parabolic 
contractions, see an explicit description in Theorem~\ref{thm:B}. We also prove that, for
all Richardson elements in question, the multiset of degrees $\{^{e\!}\cF_i\}$ coincides with the multiset of degrees of basic invariants in $\eus S(\el)^L$ (Proposition~\ref{degrees-ppy}). This provides a full description of bi-degrees of basic invariants in $\eus S(\q)^Q$.
\\  \indent
There are also `good' Richardson orbits and parabolic contractions for all simple $\g$.
The case of regular nilpotent elements, with $P=B$, is covered by our previous article \cite{alafe1},
and here we prove that Theorem~\ref{thm:main3} applies to the subregular nilpotent elements and hence to the
contractions associated with the minimal parabolic subalgebras (see Section~\ref{sect5}).
Although subregular nilpotent elements have some peculiarities if $\g$ is of type $\GR{G}{2}$, the resulting description appears to be the same for all simple Lie algebras.
Unfortunately, there are Richardson elements $e$ (at least for $\g=\son$) such that
$^{e\!}\cF_1,\ldots,\,^{e\!}\cF_l$ are algebraically dependent for any choice of $\cF_1,\dots,\cF_l$~\cite[Example\,4.1]{ppy}. This implies that $\cF_1^\bullet,\dots, \cF_l^\bullet$ are also algebraically dependent, and our technique does not apply. However, this does \un{not} necessarily mean that here $\eus S(\q)^Q$ cannot be a polynomial algebra. 

To a great extent, article \cite{ppy} was motivated by the following conjecture of Premet:
\\[.7ex]
{\sl If $e\in\g$ is a nilpotent element, then $\eus S(\g_e)^{\g_e}$ is a graded 
polynomial algebra in $l$ variables.}
\\[.7ex]
Since then, it was discovered that this conjecture is false. A counterexample, with $\g$ of type
$\GR{E}{8}$, is presented in \cite{demolish}.
Therefore, it is a challenge to classify all nilpotent elements (orbits) such that 
$\eus S(\g_e)^{\g_e}$ is graded polynomial. We hope that theory of parabolic contractions can
provide new insights on the structure of the algebras $\eus S(\g_e)^{\g_e}$ and $\eus S(\g_e)^{G_e}$
for Richardson elements $e$.

\noindent
{\sl Main notation.} \nopagebreak

-- \ the centraliser in $\g$ of $x\in\g$ is denoted by $\g_x$.  

-- \ $\varkappa$ is the Killing form on $\g$.

-- \ If $X$ is an irreducible variety, then $\bbk[X]$ is the algebra of regular functions 
and $\bbk(X)$ is the field of rational functions on $X$. If $X$ is acted upon by an algebraic 
group $A$, then $\bbk[X]^A$ and $\bbk(X)^A$ denote the subsets of respective 
$A$-invariant functions.

-- \ If $\bbk[X]^A$ is finitely generated, then $X\md A:=\spe(\bbk[X]^A)$ and the {\it
quotient morphism\/}
$\pi: X\to X\md A$ is determined by the inclusion $\bbk[X]^A \hookrightarrow \bbk[X]$.
 If  $\bbk[X]^A$ is graded polynomial, then the elements of any set of algebraically independent homogeneous generators  will be referred to as {\it basic invariants\/}.

-- \ $\eus S^i(V)$ is the $i$-th symmetric power of the vector space $V$ over $\bbk$ and $\eus S(V)=\oplus_{i\ge 0}\eus S^i(V)$ is the symmetric algebra of $V$ over $\bbk$; \ $\bbk[V]_n=\eus S^n(V^*)$ and 
$\bbk[V]=\eus S(V^*)$.

\section{Constructing invariants for parabolic contractions} 
\label{sect1}

\noindent
The Lie algebra $\q=\p\ltimes\n_-^a$ is an In\"on\"u-Wigner (= $1$-parameter)
contraction of $\g$.
In \cite[Sect.\,1]{alafe1}, we provided a general method for constructing invariants of adjoint and 
coadjoint representations of such contractions from invariants of the initial Lie algebra.
Here we recall the relevant notation and the method in the setting of parabolic contractions.

The words  `In\"on\"u-Wigner contraction' mean that the Lie bracket  in $\q$
\eqref{eq:skobka} can be obtained in the following way.
Consider the invertible linear map $\mathsf c_t: \g\to \g$, $t\in \bbk\setminus\{0\}$, 
such that $\mathsf c_t(p+\eta)=p+t\eta$ \ ($p\in\p$, $\eta\in\n_-$) and
define the new bracket  $[\ ,\ ]_{(t)}$
on the vector space $\g$ by the rule
\[
     [x,y]_{(t)}:= \mathsf c_t^{-1}\bigl( [\mathsf c_t(x), \mathsf c_t(y)]\bigr), \quad x,y\in\g \ .
\]
Write $\g_{(t)}$ for the corresponding Lie algebra.
The operator $(\mathsf c_t)^{-1}=\mathsf c_{t^{-1}}: \g\to \g_{(t)}$ yields an
isomorphism between the Lie algebras $\g=\g_{(1)}$ and $\g_{(t)}$,  hence
all algebras $\g_{(t)}$ are isomorphic.
It is easily seen that $\lim_{t\to 0}\g_{(t)}\simeq \p\ltimes (\g/\p)^a=\q$. 

To construct invariants of the coadjoint representation of $Q$, we use the decomposition
$\g=\p\oplus\n_-$ and the corresponding bi-grading 
\beq   \label{eq:nomer1}
   \bbk[\g^*]=\eus S(\g)=\bigoplus_{i,j\ge 0}\eus S^{i}(\p)\otimes \eus S^j(\n_-)
\eeq
If $\cH\in \eus S(\g)$ is homogeneous (of total degree $n$) then $\cH^\bullet$ stands for its 
bi-homogeneous component having the highest degree with respect to $\n_-$. That is, if 
$\cH=\sum_{a\le i\le b} \cH^{(n-i,i)}$,
where $\cH^{(n-i,i)}\in \eus S^{n-i}(\p)\otimes \eus S^i(\n_-)$ and
$\cH^{(n-b,b)}\ne 0$, then $\cH^\bullet:=\cH^{(n-b,b)}$. In this situation, we also set
$\deg_\p(\cH^\bullet)=n-b$ and $\deg_{\n_-}(\cH^\bullet)=b$.

\begin{thm}[\protect{\cite[Theorem\,1.1]{alafe1}}]   \label{thm:coadj-inonu}
If  $\cH\in \eus S^n(\g)^G=\bbk[\g^*]^G_n$, then $\cH^{\bullet}\in \eus S^n(\q)^Q=\bbk[\q^*]^Q_n$.
\end{thm}

\noindent
Say that $\cH^\bullet$ is the {\it highest component\/} of $\cH\in\bbk[\g^*]^G_n$
(with respect to the parabolic contraction $\g\leadsto \q$). 
Let $\eus L^\bullet(\bbk[\g^*]^G)$ denote the linear span of
$\{ \cH^\bullet \mid \cH\in\bbk[\g^*]^G \ \text{ is homogeneous}\}$.
Clearly, it 
is a graded algebra, and Theorem~\ref{thm:coadj-inonu} implies that 
$\eus L^\bullet(\bbk[\g^*]^G)\subset \bbk[\q^*]^Q$.
We say that $\eus L^\bullet(\bbk[\g^*]^G)$ is the {\it algebra of highest components\/}
for $\bbk[\g^*]^G$ (relative to bi-grading~\eqref{eq:nomer1}).

Invariants of the adjoint representation of $Q$ can be constructed in a similar (``dual'') way.
Set $\n_-^*:=\p^\perp$, the annihilator of $\p$ in $\g^*$. Likewise, $\p^*:=(\n_-)^\perp$.
Then $\g^*=\n_-^*\oplus\p^*$.
Having identified the vector spaces $\g^*$ and $\q^*$,
we play the same game with the bi-grading
\beq   \label{eq:nomer2}
   \bbk[\g]=\eus S(\g^*)=\bigoplus_{i,j\ge 0}\eus S^{i}(\p^*)\otimes \eus S^j(\n_-^*)
\eeq
and  homogeneous elements of $\eus S(\g^*)^G=\bbk[\g]^G$.  
For $\cH\in \eus S^n(\g^*)$, let $\cH_\bullet$
denote its bi-homogeneous component relative to bi-grading 
\eqref{eq:nomer2} having the highest degree with respect to $\p^*$. 

\begin{thm}[\protect{\cite[Theorem\,1.2]{alafe1}}]   \label{thm:adj-inonu}
If\/ $\cH\in \bbk[\g]^{G}_n$, then $\cH_\bullet\in \bbk[\q]^Q_n$.
\end{thm}

\noindent
Likewise, one obtains the respective algebra of highest components, $\eus L_\bullet(\bbk[\g]^G)$, 
which can be regarded as a graded subalgebra of $\bbk[\q]^Q$.

Since $\g$ is semisimple, one may identify $\g$ and $\g^*$ (and hence
$\eus S(\g)$ and $\eus S(\g^*)$) as $G$-modules using the Killing form $\varkappa$. 
Note that upon this identification $\eus S(\g)$ and $\eus S(\g^*)$, one obtains two essentially 
different bi-gradings, \eqref{eq:nomer1} and \eqref{eq:nomer2}, of one and the same algebra.
Namely, since $\n_-^*\simeq \n$ and $\p^*\simeq \p_-$, these bi-gradings are determined by the  
decompositions $\g=\p\oplus\n_-$ and $\g^*=\g=\n\oplus\p_-$, respectively. 
The upshot is that, for a given homogeneous $G$-invariant $\cH$, there are two constructions
of the ``highest bi-homogeneous component'', $\cH^\bullet$ and $\cH_\bullet$. But these highest components are determined via different bi-gradings and belong to different algebras of $Q$-invariants!

 Since $\g$ and $\q$ (as well as $\g^*$ and $\q^*$) are naturally identified as vector spaces, we always think of $\q$ and $\q^*$ as vector spaces equipped with the decompositions
\[
          \q=\p\oplus\n_- \quad \text{ and } \quad  \q^*=\n\oplus\p_- .
\]
All summands here are $P$-modules and in both cases, the second summand
is $Q$-stable.
\begin{lm}[\protect{\cite[Lemma\,1.3]{alafe1}}]   \label{lm:poincare}
The graded algebras $\bbk[\g^*]^G$ and $\eus L^\bullet(\bbk[\g^*]^G)$ have the same Poincar\'e series, i.e., 
$\dim\bbk[\g^*]^G_n=\dim\eus L^\bullet(\bbk[\g^*]^G_n)$ for all $n\in\BN$; and likewise for\/ $\bbk[\g]^G$ and $\eus L_\bullet(\bbk[\g]^G)$.
\end{lm}

By \cite[Theorem\,2.7]{coadj07}, the algebras of invariants of the adjoint and coadjoint representations of In\"on\"u-Wigner contractions are bi-graded.
The embeddings $\eus L^\bullet(\bbk[\g^*]^G)\subset \bbk[\q^*]^Q$ and
$\eus L_\bullet(\bbk[\g]^G)\subset \bbk[\q]^Q$ (together with Lemma~\ref{lm:poincare}) prompt 
the natural question whether these are equalities.
We will see in Section~\ref{sect2} that
$\eus L_\bullet(\bbk[\g]^G)\subsetneqq \bbk[\q]^Q$ unless $\el=\g$. Moreover, the $Q$-invariants 
$\cH_\bullet$ play no role in describing $\bbk[\q]^Q$.
But the situation is different for the coadjoint representation. In all cases,
when we can describe the algebra $\bbk[\q^*]^Q$,
the equality $\eus L^\bullet(\bbk[\g^*]^G)= \bbk[\q^*]^Q$ will be an important ingredient
of the final result, see Sections~\ref{sect3}--\ref{sect5}. 
Earlier, we proved that this equality holds for $\p=\be$, see~\cite[Section\,3]{alafe1}.
Similar phenomena occur also for the adjoint and coadjoint representations of
$\BZ_2$-contractions of $\g$, see \cite{rims07}.

\section{Invariants of the adjoint representation of $Q$} 
\label{sect2}

\noindent
In this section, we describe the algebra of invariants and the 
quotient morphism for the adjoint representation of any parabolic contraction of $\g$. 

\noindent
To prove that a certain set of invariants of an algebraic group action generates  the whole
algebra of invariants, we use the following variation of Igusa's lemma.  

\begin{lm}   \label{lm:igusa}
Let $A$ be a connected algebraic group acting regularly on an irreducible affine variety 
$X$. Suppose that a finitely generated
subalgebra $S\subset \bbk[X]^A$ has the following properties:
\begin{itemize}
\item[\sf (i)] \  $Y:=\spe S$ is normal;
\item[\sf (ii)] \ generic fibres of $\pi: X\to Y$ are irreducible; 
\item[\sf (iii)] \ $\dim X-\dim Y=\max_{x\in X}\dim A{\cdot}x$.
\item[\sf (iv)] \ $\Ima(\pi)$ contains an open subset\/ $\Omega$ of\/ $Y$ such that
$\codim (Y\setminus \Omega)\ge 2$.
\end{itemize}
Then $S=\bbk[X]^A$. In particular, the algebra of $A$-invariants is finitely generated.
\end{lm}
\begin{proof}
By property (iv), $\pi$ is dominant and then properties (ii) and (iii) imply that \\
\hbox to \textwidth{
\quad (ii') \hfil the fibres of $\pi$ over a dense open subset of $Y$ contain a dense $A$-orbit.\hfil}
Then properties (i), (ii'), and (iv) constitute the assumptions of Igusa's lemma, see
\cite[Lemma\,4]{ig} or \cite[Lemma\,6.1]{rims07}
\end{proof}

\begin{rmk}  \label{rmk:zamena}
If the group $A$ is unipotent, then the hypotheses of Lemma~\ref{lm:igusa} 
imply that generic fibres of $\pi$ are just $A$-orbits.
\end{rmk}

Let $\z(\el)$ denote the centre of $\el$ and $\z(\el)_{reg}=\{x\in \z(\el)\mid \g_x=\el\}$.
Then $\z(\el)_{reg}$ is a dense open subset of $\z(\el)$.

\begin{lm}  \label{lm:clear1}
If $x\in \z(\el)_{reg}$  and $n\in\n$ is arbitrary, then 
$x+n$ and $x$ belong to the same $\Ad\,N$-orbit.
\end{lm}
\begin{proof}
Clearly, $(\Ad\,N)x\subset t+\n$ for all $x\in \z(\el)$. If $x\in \z(\el)_{reg}$, then
$\dim (\Ad\,N)x=\dim \n$. It is also known that the orbits of a unipotent group acting on an 
affine variety are closed, see e.g. \cite[p.\,35]{steinb}. Hence $(\Ad\,N)x=x+\n$.
\end{proof}

\begin{prop}  \label{prop:P-L}
For any parabolic subgroup $P$ with a Levi subgroup $L$, we have 
$\bbk[\p]^P\simeq \bbk[\el]^L$. More precisely, the isomorphism $\p\md P\simeq \el\md L$ is 
induced by the projection $\p\to \p/\n\simeq \el$.
\end{prop}
\begin{proof} The projection $\tau: \p\to \p/\n\simeq \el$ is surjective and $P$-equivariant, and
the $N$-action on $\p/\n$ is trivial. It follows that the comorphism $\tau^{\#}: \bbk[\el]\to
\bbk[\p]$ yields an $L$-equivariant embedding $\bbk[\el]\hookrightarrow \bbk[\p]^N$.
For $x\in\z(\el)_{reg}$, it follows from Lemma~\ref{lm:clear1} that
$\tau^{-1}(x)=(\Ad\,N)x=x+\n$.
In particular, $\max_{x\in\el}\dim (\Ad\,N)x=\dim N$. Since {\sl all\/} the fibres of $\tau$ 
are irreducible,  Lemma~\ref{lm:igusa} applies with $\pi=\tau$, $Y=\p/\n\simeq \el$, etc., and we conclude 
that  $\bbk[\el]= \bbk[\p]^N$.
Hence $\bbk[\p]^P=(\bbk[\p]^N)^L\simeq \bbk[\el]^L$.
\end{proof}

Recall that $Q=P\ltimes N_-^a$ and $\q=\p\ltimes \n_-^a$, and our goal is to describe the  
algebra $\bbk[\q]^Q$. 

\begin{thm}   \label{thm:adj}
We have 
$\bbk[\q]^Q\simeq \bbk[\p]^P\simeq\bbk[\el]^L$. In particular, $\bbk[\q]^Q$ is a graded polynomial 
algebra. Furthermore, the quotient morphism $\pi_\q: \q\to\q\md Q\simeq \bbk^l$ is equidimensional.
\end{thm}
\begin{proof}
1) \ In view of Proposition~\ref{prop:P-L}, we have to prove that $\bbk[\q]^Q\simeq \bbk[\p]^P$.
Since $\bbk[\q]^Q=(\bbk[\q]^{N_-^a})^P$, the assertion will follow from the fact that 
\beq  \label{eq:isom-Na}
     \bbk[\q]^{N_-^a}\simeq \bbk[\p]
\eeq
and this isomorphism is compatible with the $P$-actions.
To  this end, consider the sur\-jec\-tive $Q$-equi\-va\-ri\-ant projection
$\pi: \q\to \q/\n_-^a \simeq \p$. By~Eq.~\eqref{eq:adj-action},  the subgroup 
$N_-^a\subset Q$ acts trivially on $\q/\n_-^a$. 
It follows that the comorphism $\pi^{\#}$ yields a $P$-equivariant 
embedding $\bbk[\p]\subset \bbk[\q]^{N_-^a}$. 
Again, to see that this is an equality, we use Lemma~\ref{lm:igusa}. 
If $x\in\z(\el)_{reg}\subset\p$ and $\eta\in \n_-^a$, then $x\circ \eta\in \n_-^a$ and 
$x\circ \eta=0$ if and only if $\eta=0$. Therefore, 
$(\Ad_Q N_-^a)(x)=x+\n_-^a$ (use Eq.~\eqref{eq:adj-action} with $s=1$, $\eta'=0$, and $p=x$). This implies that
   $\displaystyle \max_{\xi\in\q}\dim N_-^a{\cdot}\xi=\dim N$.
As all the fibres of $\pi$ are irreducible, 
Lemma~\ref{lm:igusa} applies here, and we conclude that Eq.~\eqref{eq:isom-Na} holds.

2) \ Gathering all previous descriptions, we see that $\pi_\q$ is the composition
\[
    \q=\p\ltimes \n_-^a \to \p\to \p/\n\simeq \el\to \el\md L .
\]
Since $\pi_\el: \el\to \el\md L$ is equidimensional \cite{ko63}, the equidimensionality of $\pi_\q$ 
follows.
\end{proof}

Comparing with the adjoint representation of $G$, we see that the algebra of 
$Q$-invariants remains polynomial, but the degrees of basic invariants somehow decrease.
This certainly means that here
$\eus L_\bullet(\bbk[\g]^G)\subsetneqq \bbk[\q]^Q=\bbk[\el]^L$ unless $\el=\g$.

\section{On invariants of the coadjoint representation of $Q$} 
\label{sect3}

\noindent
In this section, we present some general properties of the invariants of the coadjoint representation 
of $Q$. This will be a base for the explicit results on $\eus S(\q)^Q$ 
presented in Sections~\ref{sect4} and~\ref{sect5}.
The coadjoint representation is much more interesting than the adjoint one since 
$\bbk[\q^*]=\eus S(\q)$ is a Poisson algebra, $\eus S(\q)^Q$ is the centre of this Poisson algebra, 
and $\eus S(\q)$ is related to the enveloping algebra of 
$\q$ via the Poincar\'e-Birkhoff-Witt theorem. However, the coadjoint representation is also
much more complicated, and we cannot describe $\eus S(\q)^Q$ for all parabolic contraction.

Recall that $\q$ is isomorphic to $\p\oplus \g/\p \simeq \p\oplus \n_-$ as a vector space and 
a $P$-module. Then  $\q^*$ is identified with the direct sum of $P$-modules 
$\n\oplus\p_-$, where $\p_-\simeq \p^*$.

By a seminal result of R.W.\,Richardson~\cite{rwr74}, $P$ has a dense orbit in $\n$; in other words, 
there exists $e\in\n$ such that $[\p,e]=\n$. Then $\p$ is called a {\it polarisation\/} of $e$.
The nilpotent elements of $\g$ occurring as representatives 
of dense $P$-orbits in $\n=\p^{nil}$ for some $P$ (and the corresponding $G$-orbits)
are said to be {\it Richardson} or {\it polarisable}.  
If $\p=\el\oplus\n$ is a polarisation of $e$, then 
\[
    \dim G{\cdot}e=2\dim\n=\dim G/L  \ \text{ and } \ \p_e=\g_e .
\]
However, it can happen that $P_e \subsetneqq G_e$, and in such cases $G_e$ is necessarily 
disconnected. The orbit $G{\cdot}e$ depends only on the  Levi subalgebra of $\p$. That is, if  
$\p'=\el\oplus\n'$ is another parabolic subalgebra with the same Levi subalgebra $\el$ and $e'\in\n'$ is Richardson, then $G{\cdot}e=G{\cdot}e'$.

If $\g$ is of type $\GR{A}{l}$, then all nilpotent elements are Richardson, otherwise this is not 
always the case. We refer to  \cite{ke83} for an explicit description of the Richardson
elements and their polarisations for all classical Lie algebras. 

For an algebraic group $A$ with Lie algebra $\ah$,
the {\it index} of  $\ah$, $\ind\ah$, is defined as the minimal codimension of an $A$-orbit 
in the coadjoint representation. By the Rosenlicht theorem~\cite[1.6]{brion}, one also has  
$\ind\ah=\trdeg \bbk(\ah^*)^A$. The index of a reductive Lie algebra equals the rank.
It is easily seen that the index cannot decrease under 
contractions, 
hence $\ind\q\ge \ind\g=\rk\g$. 
 
\begin{thm}    \label{thm:ind-q}
We have\/ $\ind\q=\rk\g$ \ for any parabolic contraction.
\end{thm}
\begin{proof}
As $\q=\p\ltimes\n_-^a$ is a semi-direct product with an {\sl abelian\/} ideal $\n_-^a$, one can 
use Ra\"is' formula~\cite{rais} for computing $\ind\q$. Namely, 
\[
   \ind \q=\ind \p_\xi +(\dim\n - \max_{x\in\n} \dim P{\cdot}x) ,
\]
where $\n$ occurs as the dual $P$-module to $\n_-^a\simeq \g/\p$ and $\xi\in\n$ is a generic element. 
Because of the existence of Richardson elements in $\n$, 
Ra\"is' formula boils down to the equality $\ind \q=\ind \p_e=\ind\g_e$.
For all Richardson elements,  
the equality $\ind\g_e=\rk\g$ is proved in~\cite{charb-mor} (see also \cite[Cor.\,3.4]{dd03} for
the case of subregular elements).
\end{proof}

\begin{rmk}
The assertion that ``$\ind\g_e=\rk\g$ for all nilpotent elements of $\g$'' is known
as {\it the Elashvili conjecture}, see \cite[Sect.\,3]{dd03}.  Despite the extreme simplicity of the 
statement, there is no general conceptual proof as yet.
However, partial results of several authors (Yakimova, Charbonnel \& Moreau) and computer computations (de Graaf) together provide an affirmative answer to the Elashvili conjecture. 
A historic outline can be found in~\cite{charb-mor}.
\end{rmk}
 
\subsection{A sufficient condition for algebraic independence.} 
Let $\cF_1,\ldots, \cF_l$ be any set of basic invariants in $\bbk[\g]^G$.
A classical result of Kostant asserts that, for $x\in\g$, we have
$\dim \g_x=l$ if and only if the differentials 
${\textsl d}_x \cF_1,\dots,{\textsl d}_x \cF_l$ are linearly independent, see \cite[Theorem 9]{ko63}. 
As the elements $x$ of $\g$ with $\dim\g_x=\rk\g=l$ are said to be regular, this assertion is sometimes called {\it Kostant's regularity criterion}.
Having identified $\bbk[\g]$ and $\eus S(\g)$, one can rewrite this property in terms 
of $\dim\g_\xi$ ($\xi\in\g^*$) and $\{ {\textsl d}_\xi \cF_i\}_{i=1}^l$ for $\cF_1,\ldots, \cF_l\in \eus S(\g)^G$.
There is also a more fancy way to express this criterion.

\begin{df}[cf.\,{\cite[Definition\,2.2]{kys-1param}}]
Let $\ah$ be an algebraic Lie algebra  with $\ind\ah=l$
and $\cF_1,\ldots, \cF_l$ algebraically independent 
elements of $\eus S(\ah)^{\ah}$. Suppose that 
${\textsl d}_\gamma \cF_1,\dots,{\textsl d}_\gamma \cF_l$ are linearly independent if and only if
$\dim\ah_\gamma=\ind \ah$ ($\gamma\in\ah^*$).
Then we say that the polynomials $\cF_i$ ($1\le i\le l$) satisfy {\it the Kostant  equality} (in $\ah$).
\end{df}

Recall from Section~\ref{sect1} that we have defined a linear
operator $\mathsf c_t:\g\to\g$ ($t\in\bbk^\times$) and a family of Lie algebras  $\g_{(t)}$ such that 
$\lim_{t\to 0}\g_{(t)}=\q$. Having extended $\mathsf c_t$ to $\eus S(\g)$ in the obvious way, 
one can regard $\mathsf c_t(\cH)$ as an element of $\eus S(\g)[t]$ for any $\cH\in\eus S(\g)$. 
Then define $\deg_t \cH$ to be the usual degree in $t$ of $\mathsf c_t(\cH)$.
If $\deg_t \cH=d$, then the limit $\lim_{t\to 0} t^d \mathsf c_{t^{-1}}(\cH)$ exists and is a nonzero
element of $\eus S(\g)$. If $\cH$ is homogeneous, then
it follows immediately from the definition of $\mathsf c_t$ that 
$\lim_{t\to 0} t^d \mathsf c_{t^{-1}}(\cH)=\cH^\bullet$ and 
$\deg_t \cH=\deg_{\n_-}(\cH^\bullet)$.  See also \cite[Sect.\,3]{kys-1param} for a more general setup. 

Since $\ind\q=\rk\g$ (Theorem~\ref{thm:ind-q}), the theory developed in~\cite{kys-1param} yields 
the following statement in the setting of the parabolic contractions:

\begin{thm}[{\cite[Theorem~3.8]{kys-1param}}]   \label{t-degrees}
If $\cF_1,\ldots,\cF_l\in\eus S(\g)^{G}$ are the basic invariants, then 
$\sum_{i=1}^{l}\deg_{\n_-}(\cF_i^\bullet) \ge \dim\n$. Moreover, 
the equality holds if and only if $\cF_1^\bullet,\dots,\cF_l^\bullet$ are algebraically independent,
and in this case the polynomials $\{\cF_i^\bullet\}_{i=1}^l$
also satisfy the Kostant equality in $\q$.
\end{thm}

In Sections~\ref{sect4} and~\ref{sect5}, we consider certain parabolic contractions, where we are able 
to prove that  $\cF_1^\bullet,\dots,\cF_l^\bullet$ are algebraically independent. And this will finally lead us to the conclusion that $\eus S(\q)^Q=\bbk[\cF_1^\bullet,\dots,\cF_l^\bullet]=\eus L^\bullet(\eus S(\g)^G)$ is a polynomial algebra.
However, it can happen that $\cF_1^\bullet,\dots,\cF_l^\bullet$ are algebraically dependent for any choice
of the basic invariants $\cF_1,\ldots,\cF_l$ (see Remark~\ref{rmk:reason} below).

\subsection{Symmetric invariants of centralisers and contractions}
Here we explain an astonishing relationship between invariants of $(\q,\ads)$ with 
$\q=\p\ltimes\n_-^a$ and 
invariants of $(\g_e,\ads)$, where $e\in\g$ is Richardson with polarisation $\p$.

Let $e\in \g$ be a nonzero nilpotent element and
$\{e,h,f\}$ be an $\tri$-triple in $\g$ (i.e., $[e,f]=h$, $[h,e]=2e$, $[h,f]=-2f$). 
Then $\varkappa(e,f)=1$.
If $V\subset\g$ is the orthogonal complement of 
$e$ with respect to  $\varkappa$, then $\g=\bbk f\oplus V$ and $\g_e\subset V$.

\begin{lm}[\protect {\cite[Lemma\,A.1]{ppy}}]    \label{f-ort}
Let  $\cH\in\eus S(\g)^G$ be homogeneous, of degree $m$.  Consider the decomposition 
$\cH=\sum_{i\in\BN} f^{m-i}\cH_{i}$, 
where $\cH_i\in\eus S^{i}(V)$.  If $f^{m-k}\cH_{k}$ is the nonzero summand with minimal $k$,
then $\cH_k\in\eus S^k(\g_e)\subset \eus S^k(V)$. 
Moreover, this $\cH_k$ is $G_e$-invariant.
\end{lm}

\begin{rmk}  The polynomial
$\cH_k$ coincides with certain polynomial $^{e\!}\cH$ that can be constructed via the Slodowy slice associated with $\{e,h,f\}$, see \cite[Prop.\,0.1\ \&\ Cor.\,A.2]{ppy}.
Therefore, we use notation $^{e\!}\cH$ for the above polynomial $\cH_k$ associated with $\cH$.
\end{rmk}

From now on, we assume that $e\in\g$ is Richardson and $\p=\el\oplus\n$ is a polarisation of $e$. 
To any homogeneous $\cH\in\eus S(\g)^G$,
we have attached two polynomials, $\cH^\bullet \in \eus S(\q)^Q$ (see Section~\ref{sect1})
and $^{e\!}\cH\in \eus S(\g_e)^{G_e}$. To provide a link between them,
one has to adjust the above construction of $^{e\!}\cH$ to 
the decompositions $\q=\p\oplus\n_-$ and $\q^*=\n\oplus\p_-$. 
If $f\in \n_-$, then no special adjustment is needed. Otherwise, we take $y\in \n_-$
such that $\varkappa(y,e)=1$.
Then $f-y\in V$ and   $\g=\bbk y\oplus V$. 
If $\cH=\sum_i y^{m-i} \tilde \cH_i$ with $\tilde \cH_i\in\eus S^{i}(V)$, 
then it is  easily seen that the nonzero summand with minimal $i$
occurs for $i=k$ and $\tilde \cH_k=\cH_k$. 
That is, we can (and will) use such an $y\in\n_-$ in place of $f$ for defining $^{e\!}\cH$.

Recall that $\q^* =\n\oplus \p_-$ is a sum of $P$-modules and the second summand here
is also a $Q$-submodule.
Since $e\in\n$,  $e+\p_-$ is an affine subspace of $\q^*$ and
we identify $\bbk[e+\p_-]$ with $\bbk[\p_-]=\eus S(\p)$. Recall also that
$\g_e=\p_e\subset\p$ and therefore $\eus S(\g_e)\subset \eus S(\p)$. Hence $^{e\!}\cH$ can be 
thought of as element of $\eus S(\p)$ that belongs to the subalgebra $\eus S(\g_e)^{G_e}$.

\begin{prop}               \label{coincidence}
If $\cH\in\eus S(\g)^G$ is homogeneous, then under the above identifications, we have
$\cH^\bullet\vert_{e+\p_-}=\,^{e\!}\cH$. In particular, $\cH^\bullet\vert_{e+\p_-}$ belongs to the subalgebra 
$\eus S(\g_e)^{G_e}$ of\/ $\eus S(\p)$ and $\deg_\p(\cH^\bullet)=\deg (^{e\!}\cH)$.
\end{prop}
\begin{proof}
As in Lemma~\ref{f-ort}, assume that $\deg(\cH)=m$ and $\deg (^{e\!}\cH)=k$.
By definition, $\cH^\bullet$ is the nonzero bi-homogeneous component of $\cH$ with highest 
$\n_-$-degree.
By Theorem~\ref{thm:coadj-inonu}, $\cH^\bullet$ is $Q$-invariant and hence $P$-invariant.
Since $\ov{P{\cdot}e}=\n$, we have $\overline{P(e+\p_-)}=\q^*$, and  $\cH^\bullet$ does not vanish
on $e+\p_-$.
In view of  the inclusion $\g_e\subset\p$, Lemma~\ref{f-ort} (with $f$ replaced by $y$) shows 
that  $\cH$ has a non-zero summand of 
bi-degree $(k,m-k)$ with respect to  $(\p,\n_-)$. 
Hence $\deg_{\n_-}(\cH)\ge m-k$. On the other hand, $\p\subset V$ and 
$\cH^\bullet\vert_{e+\p_-}$ has degree at least $k$ as element of $\eus S(\p)$.
Therefore $\deg_{\n_-}\cH = m-k$
and $\cH^\bullet\vert_{e+\p_-}=\,^{e\!}\cH$.
\end{proof}

Since $\p_-$ is a $Q$-submodule of $\q^*$, the affine subspace $e+\p_-$  is 
invariant with respect to the subgroup $P_e\ltimes N_-^a\subset Q$.
Therefore, one has a well-defined  homomorphism 
\[
\psi:\ \eus S(\q)^Q\to \bbk[e+\p_-]^{P_e\ltimes N_-^a} 
\]
that takes $\cH$ to $\cH\vert_{e+\p_-}$.
Because $\ov{Q{\cdot}(e+\p_-)}=\q^*$,   $\psi$  is injective. It is also clear that, under the identification
$\bbk[e+\p_-]\simeq \eus S(\p)$, the image of $\psi$ belong to $\eus S(\p)^{P_e}$.
Hence $\psi$ can be thought of as a homomorphism from $\eus S(\q)^Q$ to
$\eus S(\p)^{P_e}$.

\begin{thm}   \label{thm:main3}
Let $\cF_1,\dots, \cF_l\in \eus S(\g)^G$ be the basic invariants and $\q$ the parabolic contraction 
of\/ $\g$ defined by\/ $\p$. Then 
\beq   \label{eq:alg-indep-equiv}
\text{
 $\cF_1^\bullet,\dots, \cF_l^\bullet$ are algebraically independent  if and only if \
 $^{e\!}\cF_1,\ldots,\,^{e\!}\cF_l$ are.}
\eeq
If the equivalent conditions of \eqref{eq:alg-indep-equiv} hold,
then 
\begin{itemize}
\item[\sf (i)] \ $\eus S(\q)^Q\supset \bbk[\cF_1^\bullet,\dots, \cF_l^\bullet]$ is an algebraic extension;
\item[\sf (ii)] \  $\psi(\eus S(\q)^Q) \subset \eus S(\g_e)^{P_e}$;
\item[\sf (iii)] \ Moreover, if\/ $\eus S(\g_e)^{G_e}=\eus S(\g_e)^{P_e}$ and this
algebra is freely generated  by $^{e\!}\cF_1,\ldots,\,^{e\!}\cF_l$, then 
$\eus S(\q)^Q$ is freely generated by  $\cF_1^\bullet,\dots, \cF_l^\bullet$.
\end{itemize}
\end{thm}
\begin{proof}
The equivalence of two conditions in~\eqref{eq:alg-indep-equiv} follows from the fact that
$\psi$ is injective and $\psi(\cF_i^\bullet)=^{e\!\!}\! \cF_i$ (see Proposition~\ref{coincidence}).

(i) \ 
Since $\ind\q=l$ (see Theorem~\ref{thm:ind-q}), one always has $\trdeg\, \eus S(\q)^Q\le l$. 
Therefore $\trdeg\,\eus S(\q)^Q= l$, and the assertion follows.

(ii) \ It follows from (i) that  $\psi(\eus S(\q)^Q)\supset \bbk[^{e\!}\cF_1,\ldots,\,^{e\!}\cF_l]$ 
is an algebraic extension. 
Because $^{e\!}\cF_i\in \eus S(\g_e)^{G_e}$ and 
$\eus S(\g_e)$ is algebraically closed in $\eus S(\p)$, 
we have  
\beq   \label{eq:bullet-e2}
   \psi(\eus S(\q)^Q) \subset \eus S(\g_e)\cap \eus S(\p)^{P_e}=\eus S(\g_e)^{P_e} .
\eeq

(iii) \ Here we have
$\eus S(\g_e)^{G_e}=\bbk[^{e\!}\cF_1,\ldots,\,^{e\!}\cF_l]\subset \psi(\eus S(\q)^Q)\subset \eus S(\g_e)^{P_e}
=\eus S(\g_e)^{G_e}$.
\\
Whence $\psi(\eus S(\q)^Q)=\bbk[^{e\!}\cF_1,\ldots,\,^{e\!}\cF_l]$ and therefore
$\eus S(\q)^Q=\bbk[\cF_1^\bullet,\dots, \cF_l^\bullet]$.
\end{proof}

\begin{rmk}   \label{rm:pro-Ge-i-Pe}
In view of the above theorem, it is important to know when the $P_e$- and $G_e$-invariants
in $\eus S(\g_e)$ coincide. This condition is weaker than the coincidence of $P_e$ and $G_e$.
For any Richardson element $e\in\n=\p^{nil}$, one can consider the chain of groups
\[
    G_e^o \subset P_e \subset G_e ,
\]
where $G_e^o$ is the identity component of $G_e$, and the corresponding chain of rings 
of invariants
\[
   \eus S(\g_e)^{G_e}\subset \eus S(\g_e)^{P_e}\subset \eus S(\g_e)^{G_e^o}=:\eus S(\g_e)^{\g_e} . 
\]
All these inclusions can be strict. As is well known, the equality $P_e=G_e$ has the following
geometric meaning.  The cotangent bundle $T^*(G/P)\simeq G\times_P\n$ has the natural collapsing
\[
    \phi : G\times_P\n \to G{\cdot}\n=\ov{G{\cdot}e} \subset \g 
\]
such that the fibre $\phi^{-1}(e)$ has cardinality $\#(G_e/P_e)$, see e.g. \cite[\S\,7]{bk79}.
Therefore, $\phi$ is birational
(and thereby is a resolution of singularities of $\ov{G{\cdot}e}$) if and only if $G_e=P_e$.
It is also known that  if $e$ is even (i.e., the weighted Dynkin diagram of $e$ has only even labels),
then $P_e=G_e$ for the Dynkin-Jacobson-Morozov parabolic subalgebra associated with $e$.
\end{rmk}

\section{Parabolic contractions for classical Lie algebras}
\label{sect4}

\noindent
In this section, we prove that (i)  if $\g$ is a simple Lie algebra
of type $\GR{A}{l}$ or $\GR{C}{l}$, then
$\eus S(\q)^Q$ is a polynomial algebra for any $\p$; \ (ii) if $\g$ is of type  $\GR{B}{l}$, then
$\eus S(\q)^Q$ is a polynomial algebra whenever the  
Levi subalgebra of $\p$ is of the form
$\mathfrak{gl}_{n_1}\oplus\ldots\oplus \mathfrak{gl}_{n_t}$, where $n_1,\dots,n_t$ are odd.

It is quite common in invariant-theoretic problems
that a certain method works well in types $\GR{A}{l}$ and 
$\GR{C}{l}$ and does not extend in full generality to other simple Lie algebras. This 
has happened in~\cite{ppy} in connection with the study of symmetric invariants of centralisers,
and also evinces here, because our approach relies on results of that paper.
The same phenomenon also manifests itself in \cite{ffp, ffp2}, 
where an explicit description of a Borel (or parabolic) contraction of simple $SL_{l+1}$ or 
$Sp_{2l}$-modules  to  ``highest weight" $Q$-modules is obtained. 
No similar results are known so far in other types.

\begin{thm}                       \label{thm:AC}
Suppose that $\g$ is either $\gt{sl}_{l+1}$ or $\gt{sp}_{2l}$ and $\q$ is a parabolic contraction of\/ $\g$.
Then there exist basic invariants $\cF_1,\dots, \cF_l\in \eus S(\g)^G$ such that
$\cF_1^\bullet,\dots, \cF_l^\bullet$ freely generate $\eus S(\q)^{Q}$ and satisfy the Kostant 
equality in $\q$, and the equality $\sum_{i=1}^l\deg_{\,\n_-}(\cF_i^\bullet)=\dim\n$ holds.
\end{thm}
\begin{proof}  
For $\g=\gt{sl}_{l+1}$ or $\gt{sp}_{2l}$, let  $\cF_1,\dots, \cF_l\in \eus S(\g)^G$ be 
the coefficients of the characteristic polynomial of a matrix in $\g$. It is proved in 
\cite[Theorems~4.2\,\&\,4.4]{ppy} that
$\,^{e\!}\cF_1,\ldots,\,^{e\!}\cF_l$ are algebraically independent  for  {\sl every\/} nonzero 
nilpotent element $e\in\g$ and
\beq     \label{eq:e-adapted-inv}
    \eus S(\g_e)^{G_e}=\eus S(\g_e)^{\g_e}=\bbk[^{e\!}\cF_1,\ldots,\,^{e\!}\cF_l] .
\eeq
This also shows that if $e\in\n$ is Richardson, then $ \eus S(\g_e)^{G_e}= \eus S(\g_e)^{P_e}$.
Therefore, applying Theorem~\ref{thm:main3} to our Richardson element $e\in \n$, we conclude that, 
for the above-mentioned choice of basic invariants,
 $\eus S(\q)^{Q}$ is freely generated by $\cF_1^\bullet,\dots, \cF_l^\bullet$.

Since $\cF_1^\bullet,\dots,\cF^\bullet_l$ are algebraically independent, it follows from
Theorem~\ref{t-degrees} that  $\cF_1^\bullet,\dots,\cF^\bullet_l$ satisfy the Kostant equality and
$\sum_{i=1}^l\deg_{\,\n_-}(\cF_i^\bullet)=\dim\n$.
\end{proof}

To describe our results in the orthogonal case, we introduce some terminology.
We say that a parabolic subalgebra of $\g=\gt{so}_{2l+1}$ is {\it admissible}, if  
the  Levi subalgebras of\/ $\p$ are of the 
form\/ $\mathfrak{gl}_{n_1}\oplus\ldots\oplus \mathfrak{gl}_{n_t}$,
where $n_1,\ldots,n_t$ are odd. The corresponding parabolic contractions and Richardson orbits are
said to be {\it admissible}, too.

\begin{thm}                       \label{thm:B}
Let\/ 
$\q$ be an admissible parabolic contraction of\/ $\gt{so}_{2l+1}$. 
Then there exist basic invariants $\cF_1,\dots, \cF_l\in \eus S(\g)^G$ such that
$\cF_1^\bullet,\dots, \cF_l^\bullet$ \ freely generate $\eus S(\q)^{Q}$
and satisfy the Kostant equality in $\q$, and the equality $\sum_{i=1}^l\deg_{\,\n_-}(\cF_i^\bullet)=\dim\n$ holds.
\end{thm}
\begin{proof} As in Theorem~\ref{thm:AC}, take $\cF_1,\dots, \cF_l$ to be the coefficients of the
characteristic polynomial of a matrix in $\g$. 
Suppose that $e\in\g$ is given by the partition
$\boldsymbol{\lb}=(\lb_1\ge\lb_2\ge\dots\ge\lb_t)$ of $2l+1$ such that $\lb_1$ is odd and all other parts are 
even. (Recall that, for the nilpotent elements of $\sov$, each even part of \
$\boldsymbol{\lb}$ occurs an even number of times.)
By \cite[Theorem~4.7]{ppy}, $\cF_1,\dots, \cF_l$ is a ``very good generating system'' for $e$,
which, in view of \cite[Theorem\,2.2]{ppy}, implies that  $\,^{e\!}\cF_1,\ldots,\,^{e\!}\cF_l$ are 
algebraically independent  and
\eqref{eq:e-adapted-inv}  holds. 

Using \cite[4.2]{ke83}, one verifies that the above elements $e$ 
are exactly the admissible Richardson elements.
In this case, $\boldsymbol{\widehat\lb}$ is of the form
$(\mu_1^{2},\mu_2^{2},\dots,\mu_s^{2},1^{2k+1})$, where all 
$\mu_i$ are odd and $\mu_1\ge \mu_2\ge \ldots \ge \mu_s\ge 3$, and then
$\el=\mathfrak{gl}_{\mu_1}\oplus\ldots\oplus \mathfrak{gl}_{\mu_t}\oplus (\mathfrak{gl}_{1})^k$. 
Therefore, Theorem~\ref{thm:main3} can be applied to the admissible
parabolic subalgebras $\p$ and parabolic contractions $\q$, and we conclude that
 $\eus S(\q)^{Q}$ is freely generated by $\cF_1^\bullet,\dots, \cF_l^\bullet$.
 The rest is the same as in Theorem~\ref{thm:AC}.
\end{proof}

\begin{rema} 
Theorem~4.7 in \cite{ppy}, which is used in the previous proof, refers also to similar 
nilpotent elements in $\mathfrak{so}_{2l}$. But all those elements are not Richardson.
\end{rema}

For the parabolic contractions described in Theorems~\ref{thm:AC} and \ref{thm:B}, the basic invariants in $\eus S(\q)^Q$ have the 
same degrees as the basic invariants in $\eus S(\g)^G$. But the algebra $\eus S(\q)^Q$ is bi-graded, 
and our next goal is to determine the bi-degrees of 
$\cF_1^\bullet,\dots,\cF^\bullet_l$ 
with respect to decomposition~\eqref{eq:nomer1} in the corresponding cases. 
By Proposition~\ref{coincidence}, we have $\deg_\p(\cF_i^\bullet)=\deg (^{e\!}\cF_i)$.
For \un{all} (resp. \un{some}) nilpotent elements in $\gt{sl}_{l+1}$ or $\gt{sp}_{2l}$ (resp.  
$\gt{so}_{2l+1}$), there is an explicit algorithm for computing the degrees of 
$\,^{e\!}\cF_1,\dots,\,^{e\!}\cF_l$ \cite[Sect.\,4]{ppy}. We prove below that, for the (admissible) Richardson elements, this can be restated in terms of a Levi subalgebra $\el\subset\p$.

\begin{prop}           \label{degrees-ppy}
Let $e\in\g$ be a Richardson element  with a polarisation $\p=\el\oplus\n$, where 
\begin{itemize}
\item $\p$ is any parabolic subalgebra, if\/ $\g=\gt{sl}_{l+1}$ or\/ $\gt{sp}_{2l}$;
\item $\p$ is admissible, if\/ $\g=\gt{so}_{2l+1}$.
\end{itemize}
Then the multiset of degrees of 
$\,^{e\!}\cF_1,\dots,^{e\!}\cF_l$  is the same as the multiset of degrees of the basic $L$-invariants in 
$\eus S(\gt l)$. 
\end{prop}
\begin{proof} 1) 
To simplify exposition, we work here with 
$\mathfrak{gl}_{l+1}$ in place of $\mathfrak{sl}_{l+1}$. 
Then 
$\{\deg \cF_j\}=\{1,2,\dots,l+1\}$.
Recall that all nilpotent elements of $\mathfrak{gl}_{l+1}$ are Richardson.
Let $e\in\mathfrak{gl}_{l+1}$ correspond to the partition $\boldsymbol{\lb}=(\lb_1,\dots,\lb_t)$, 
where $\sum_i \lb_i=l+1$ and $\lb_1\ge \lb_2\ge\dots\ge\lb_t>0$.
Then 
\[
\#\{j \mid \deg (^{e\!}\cF_j)=i\}=\lb_i, 
\]
see \cite[p.\,368]{ppy}, 
i.e., the  multiset of degrees of the $\,^{e\!}\cF_j$'s  is $\{1^{\lb_1}, 2^{\lb_2},\dots, t^{\lb_t}\}$.  
More precisely,  if 
$\lb_1+\ldots +\lb_{i-1}+1 \le \deg(\cF_j) \le \lb_1+\ldots +\lb_{i}$, then $\deg(^{e\!}\cF_j)=i$. 
On the other hand, if $\boldsymbol{\widehat\lb}=(\hat\lb_1,\dots,\hat\lb_s)$ is  the dual partition,
then 
$\el\simeq\mathfrak{gl}_{\hat\lb_1}\oplus\dots\oplus \mathfrak{gl}_{\hat\lb_s}$. Therefore,
the basic invariants of degree $i$ in $\eus S(\el)^L$ occur with multiplicity \ $\#\{j \mid \hat\lb_j\ge i\}=\lb_i$.

2) If $\g=\mathfrak{sp}_{2l}$, then $\{\deg \cF_j\}=\{2,4,\dots,2l\}$ and
 there is a similar algorithm to determine 
$\deg (^{e\!}\cF_j)$ for all $e$ \cite[4.3]{ppy}. Let $\boldsymbol{\lb}=(\lb_1,\dots,\lb_t)$ be the partition
of $2l$ corresponding to $e$. (Recall that, for the nilpotent elements of $\spv$, each odd part occurs an even number of times.) 
Then we have $\deg(^{e\!}\cF_j)=i$, if
$
     \lb_1+\ldots +\lb_{i-1}+1 \le \deg(\cF_j) \le \lb_1+\ldots +\lb_{i} .
$

\noindent
By~\cite[4.1]{ke83}, $e$ is Richardson if and only if $\boldsymbol{\lb}$ satisfies the following 
conditions:
\begin{enumerate}
\item either all $\lb_i$ are even, or $r:=\max\{j\mid \lb_j \text{ is odd}\}$  is even \\
(set $r=0$ if all parts are even);
\item $\lb_{2j-1},\lb_{2j}$ have the same parity for $2j\le r$;
\item if $\lb_{2j},\lb_{2j+1}$ are even (for $2j < r$), then  $\lb_{2j}\ge\lb_{2j+1}+2$.
\end{enumerate}

\noindent
For the Richardson elements, the above algorithm for finding $\deg(^{e\!}\cF_j)$
can graphically be presented via the chessboard 
filling of the Young diagram of
$\boldsymbol{\lb}$. See the left figure below, where $\boldsymbol{\lb}=(6,6,5,5,2)$ and 
the parts $\lb_i$ represent the columns of the diagram. For this diagram, one obtains
$\{\deg(^{e\!}\cF_j)\}=\{1^3,2^3,3^2,4^3,5\}$.

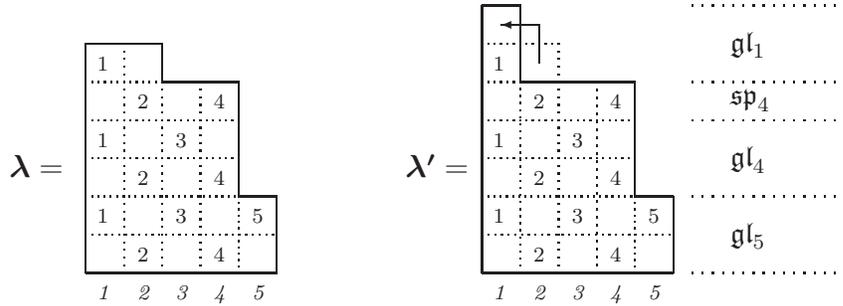
\begin{figure}[htb]
\setlength{\unitlength}{0.02in}
\begin{center}
\begin{picture}(60,55)(0,0)

\put(-20,25){$\boldsymbol{\lb}=$}
\put(0,0){\line(1,0){50}}    
\put(40,20){\line(1,0){10}}
\put(20,50){\line(1,0){20}}
\put(0,60){\line(1,0){20}}
\put(0,0){\line(0,1){60}}    
\put(20,50){\line(0,1){10}}
\put(40,20){\line(0,1){30}}
\put(50,0){\line(0,1){20}}

\qbezier[28](0,10),(24.5,10),(49,10)
\qbezier[24](0,20),(19.5,20),(39,20)
\qbezier[24](0,30),(20,30),(40,30)
\qbezier[24](0,40),(20,40),(40,40)
\qbezier[12](0,50),(10,50),(20,50)

\qbezier[30](10,0),(10,30),(10,60)
\qbezier[25](20,0),(20,25),(20,50)
\qbezier[25](30,0),(30,25),(30,50)
\qbezier[12](40,0),(40,10),(40,20)

\put(3,53){$\scriptstyle 1$}
\put(3,33){$\scriptstyle 1$}
\put(3,13){$\scriptstyle 1$}
\put(3,-7){$\scriptstyle \it 1$}

\put(13.5,43){$\scriptstyle 2$}
\put(13.5,23){$\scriptstyle 2$}
\put(13.5,3){$\scriptstyle 2$}
\put(13.5,-7){$\scriptstyle \it 2$}

\put(23.5,33){$\scriptstyle 3$}
\put(23.5,13){$\scriptstyle 3$}
\put(23.5,-7){$\scriptstyle \it 3$}

\put(33.5,43){$\scriptstyle 4$}
\put(33.5,23){$\scriptstyle 4$}
\put(33.5,3){$\scriptstyle 4$}
\put(33.5,-7){$\scriptstyle \it 4$}

\put(43.5,13){$\scriptstyle 5$}
\put(43.5,-7){$\scriptstyle \it 5$}

\end{picture}
\qquad
\begin{picture}(80,65)(-25,0)

\put(-20,25){$\boldsymbol{\lb'}=$}
\put(0,0){\line(1,0){50}}    
\put(40,20){\line(1,0){10}}
\put(10,50){\line(1,0){30}}
\put(0,70){\line(1,0){10}}
\put(0,0){\line(0,1){70}}    
\put(10,50){\line(0,1){20}}
\put(40,20){\line(0,1){30}}
\put(50,0){\line(0,1){20}}
  \put(15,55){\line(0,1){10}}
  \put(15,65){\vector(-1,0){10}}

\qbezier[28](0,10),(24.5,10),(49,10)   
\qbezier[24](0,20),(19.5,20),(39,20)
\qbezier[24](0,30),(20,30),(40,30)
\qbezier[24](0,40),(20,40),(40,40)
\qbezier[6](0,50),(5,50),(10,50)
\qbezier[12](0,60),(10,60),(20,60)

\qbezier[15](55,00),(75,00),(95,00)  \put(65,8){\small $\mathfrak{gl}_5$}
\qbezier[15](55,20),(75,20),(95,20)  \put(65,28){\small $\mathfrak{gl}_4$}
\qbezier[15](55,40),(75,40),(95,40)  \put(65,44){\small $\mathfrak{sp}_4$}
\qbezier[15](55,50),(75,50),(95,50)  \put(65,58){\small $\mathfrak{gl}_1$}
\qbezier[15](55,70),(75,70),(95,70)

\qbezier[27](10,0),(10,25),(10,50)   
\qbezier[30](20,0),(20,30),(20,60)
\qbezier[24](30,0),(30,25),(30,50)
\qbezier[12](40,0),(40,10),(40,20)

\put(3,53){$\scriptstyle 1$}
\put(3,33){$\scriptstyle 1$}
\put(3,13){$\scriptstyle 1$}
\put(3,-7){$\scriptstyle \it 1$}

\put(13.5,43){$\scriptstyle 2$}
\put(13.5,23){$\scriptstyle 2$}
\put(13.5,3){$\scriptstyle 2$}
\put(13.5,-7){$\scriptstyle \it 2$}

\put(23.5,33){$\scriptstyle 3$}
\put(23.5,13){$\scriptstyle 3$}
\put(23.5,-7){$\scriptstyle \it 3$}

\put(33.5,43){$\scriptstyle 4$}
\put(33.5,23){$\scriptstyle 4$}
\put(33.5,3){$\scriptstyle 4$}
\put(33.5,-7){$\scriptstyle \it 4$}

\put(43.5,13){$\scriptstyle 5$}
\put(43.5,-7){$\scriptstyle \it 5$}

\end{picture}
\end{center}
\end{figure}

\noindent
To describe a Levi subalgebra $\el$ corresponding to such a $\boldsymbol{\lb}$, one proceeds as
follows. Take all even pairs $\lb_{2j-1},\lb_{2j}$ ($2j\le r$) and replace them with 
$\lb_{2j-1}+1,\lb_{2j}-1$. Because of the conditions, one obtains a partition $\boldsymbol{\lb'}$ of
the form 
$\boldsymbol{\lb'}=(\underbrace{\lb'_1,\dots,\lb'_r}_{odd},\underbrace{\lb_{r+1},\dots,\lb_t}_{even})$. Note that $\lb'_r=\lb_r$. The dual partition 
$\widehat{\boldsymbol{\lb'}}=\{\hat\lb_1,\dots,\hat\lb_s\}$
determines one of the Levi subalgebras corresponding to $e$. Namely, 
$\#\{j \mid \hat\lb_j=i\}$ is even unless $i=r$, and
each pair of parts of equal size $\hat\lb_j$ gives rise to the summand $\mathfrak{gl}_{\hat\lb_j}$ in $\el$. The only non-paired part of size $r$ gives rise to the summand
$\mathfrak{sp}_r$ in $\el$.
We may think of parts of $\widehat{\boldsymbol{\lb'}}$ as the rows of ${\boldsymbol{\lb'}}$.
Then the consecutive pairs of equal rows  below or above  the level $\lb_r$ represent
the summands of the form $\mathfrak{gl}_{\hat\lb_j}$, and our graphical algorithm shows that 
the corresponding pair of rows contain boxes filled with numbers $1,2,\dots,\hat\lb_j$; while
the remaining row of length $r$ at level $\lb_r$ contains numbers $2,4,\dots,r$. 
It is important that the passage from $\boldsymbol{\lb}$ to $\boldsymbol{\lb'}$
consists in moving only empty boxes! 
(See the right figure above, where $r=4$ and $\el$ is equal to $\mathfrak{gl}_5\oplus
\mathfrak{gl}_4\oplus\mathfrak{gl}_1\oplus\mathfrak{sp}_4$.) This shows that the assertion holds 
for this specific Levi subalgebra associated with $e$. By \cite{ke83}, all other Levi subalgebras
(if any) are obtained by the following alterations: If $\el$ contains the summands 
$\mathfrak{sp}_r\oplus\mathfrak{gl}_{r+2}$, then they can be replaced with
$\mathfrak{sp}_{r+2}\oplus\mathfrak{gl}_{r+1}$ (all other summands remain intact).
Clearly, this step does not change the degrees of basic invariants in $\eus S(\el)^L$.

3) If $\g=\mathfrak{so}_{2l+1}$, then $\{\deg \cF_j\}=\{2,4,\dots,2l\}$.
For the admissible Richardson elements $e$, the algorithm for computing $\deg (^{e\!}\cF_j)$ is the same as in part 2), see~\cite[4.4]{ppy}. 
If $\boldsymbol{\lb}=(\lb_1,\dots,\lb_t)$ is admissible, i.e.,
$\lb_1$ is odd and all other parts are even, then  $\lb_2=\lb_3, \lb_4=\lb_5$, etc., and we obtain
\[
\#\{j \mid \deg (^{e\!}\cF_j)=i\}=\left\{\begin{array}{cl} [\lb_1/2], &  i=1 \\
\lb_i/2, &  i> 1 . \end{array}\right.
\]
Then $\boldsymbol{\widehat\lb}=(\hat\lb_1,\dots,\hat\lb_{2s+2k+1})$, where
$\lb_1=2s+2k+1$, $\lb_2=2s$, 
$\hat\lb_{2i-1}=\hat\lb_{2i}$ for $i=1,\dots,s$ and $\hat\lb_{2s+1}=\dots =\hat\lb_{2k+2s+1}=1$.
In this case, $\el=\mathfrak{gl}_{\hat\lb_2}\oplus\mathfrak{gl}_{\hat\lb_4}\oplus\dots\oplus \mathfrak{gl}_{\hat\lb_{2s}}\oplus (\mathfrak{gl}_1)^k$.
Therefore, the basic invariants of degree $i$ in $\eus S(\el)^L$ occur with multiplicity \ 
\[
    \left\{\begin{array}{rl}   s+k, & i=1 \\
    \#\{j \mid \hat\lb_{2j}\ge i\}, & i >1 .  \end{array} \right.
\]
It remains to observe that $s+k=[\lb_1/2]$ and, for $i>1$, we have
$\#\{j \mid \hat\lb_{2j}\ge i\}=\frac{1}{2}\#\{j \mid \hat\lb_{j}\ge i\}=
\frac{1}{2}\lb_i$.
\end{proof}

By Theorems~\ref{thm:AC} and \ref{thm:B}, the sum of $\deg_{\n_-}(\cF_i^\bullet)$ equals 
$\dim\n$. However, Proposition~\ref{degrees-ppy}
provides another approach to this equality. 

\begin{cor}              \label{cor:sum-degrees}
For the bi-homogeneous basic invariants $\cF_1^\bullet,\dots, \cF_l^\bullet$,  we have
$\sum_i\deg_{\n_-}(\cF_i^\bullet)=\dim\n$ and 
$\sum_i\deg_{\p}(\cF^\bullet_i)=\dim \be(\el)$, where $\be(\el)$ is
a Borel subalgebra of\/ $\el$.
\end{cor}
\begin{proof}
Since $\deg_\p(\cF^\bullet_i)=\deg (^{e\!}\cF_i)$, the second equality follows immediately from the
proposition. The rest follows from the equalities 

$\deg(\cF_i)=\deg(\cF^\bullet_i)$, \ 
$\sum_{i=1}^l \deg (\cF_i)=\dim\be$, \ 
 \ 
and $\dim\be=\dim\be(\el)+\dim\n_-$.
\end{proof}

\begin{ex}
{\sl 1)}  $\boldsymbol{\lb}=(6,4,2)$ determines a Richardson element in $\mathfrak{sp}_{12}$.
Here $\{\deg(\cF_i)\}=\{2,4,6,8,10,12\}$ and the algorithm transforms these degrees in 
$\{\deg(^{e\!}\cF_i)\}=\{1,1,1,2,2,3\}$. This is in accordance with the fact that
the corresponding Levi subalgebra is $\mathfrak{gl}_3\oplus\mathfrak{gl}_2\oplus\mathfrak{gl}_1$. 
Thus,  the bi-degrees $(\deg_{\p}\cF_i^\bullet, \deg_{\n_-}\cF_i^\bullet)$
of $\{\cF_i^\bullet\}$ are:
\[ 
    (1,1), (1,3), (1,5), (2,6), (2,8), (3,9) .
\]

{\sl 2)} $\boldsymbol{\lb}=(3,3,1,1)$ determines a Richardson element in $\mathfrak{sp}_{8}$.
Here $\{\deg(\cF_i)\}=\{2,4,6,8\}$ and the algorithm transforms these degrees in 
$\{\deg(^{e\!}\cF_i)\}=\{1,2,2,4\}$.
Accordingly, the corresponding Levi subalgebra is $\mathfrak{sp}_4\oplus\mathfrak{gl}_2$.
Thus,  the bi-degrees $(\deg_{\p}\cF_i^\bullet, \deg_{\n_-}\cF_i^\bullet)$
of $\{\cF_i^\bullet\}$ are: $ (1,1), (2,2), (2,4), (4,4) .$

{\sl 3)} $\boldsymbol{\lb}=(5,4,4,2,2)$ determines an admissible  Richardson element in $\mathfrak{so}_{17}$.
Here $\{\deg(\cF_i)\}=\{2,4,6,8,10,12,14,16\}$ and the algorithm transforms these numbers in 
$\{\deg(^{e\!}\cF_i)\}=\{1^2,2^2,3^2,4,5\}$. This corresponds to the fact that $\el=
\mathfrak{gl}_5\oplus\mathfrak{gl}_3$.
\end{ex}

\begin{rmk}     \label{rmk:reason}
The reason for our partial success is that there is a general relationship between $\cH^\bullet$ and
${^e\!}\cH$ (Prop.~\ref{coincidence}) and  the polynomials
$^{e\!}\cF_1,\ldots,\,^{e\!}\cF_l$  are algebraically independent for  all (resp. admissible)
Richardson elements $e$
in $\gt{sl}_{l+1}$ and $\gt{sp}_{2l}$ (resp. $\gt{so}_{2l+1}$).
However, for $\g=\mathfrak{so}_{2l}$, there are Richardson elements $e$ such that 
$^{e\!}\cF_1,\ldots,\,^{e\!}\cF_l$  are algebraically dependent for any choice of basic invariants $\cF_i$. Namely, this happens for  $e\in\mathfrak{so}_{12}$ corresponding to the partition 
$(5,3,2,2)$, see \cite[Example~4.1]{ppy}. (Here $\dim\g_e=18$ and the semisimple part of $\el$ is of type $\GR{A}{3}$.)
For the corresponding parabolic contraction $\q$, 
$\cF_1^\bullet,\dots, \cF_l^\bullet$ are also algebraically dependent, 
see~\eqref{eq:alg-indep-equiv}. One can prove that $\eus S(\q)^Q$ always has the transcendence 
degree $l$, hence here $\eus S(\q)^Q$ is not generated by 
$\cF_1^\bullet,\cdots,\cF_l^\bullet$. 
However, this does \un{not} necessarily mean that here $\eus S(\g_e)^{\g_e}$ or 
$\eus S(\q)^Q$ cannot be a polynomial algebra.
\end{rmk}

\section{Minimal parabolic subalgebras and subregular contractions}
\label{sect5}

\noindent
In this section $\g$ is a simple Lie algebra.
Fix a triangular decomposition $\g=\ut_-\oplus\te\oplus\ut$, where $\te$ is a Cartan subalgebra and
$\be=\te\oplus\ut$. Then $\Delta$ is the root system of $(\g,\te)$, $\Delta^+$ is the set of roots of 
$\ut$, $\Pi$ is the set of simple roots in $\Delta^+$, and $\delta$ is the highest root in $\Delta^+$.
Write $\g^\gamma$ for the root space corresponding to $\gamma\in \Delta$.

Let $\p$ be a minimal parabolic subalgebra of $\g$, i.e., $\dim\p=\dim\be+1$ and
$[\el,\el]\simeq \tri$.
We assume that $\p=\be\oplus\g^{-\ap}$ for some $\ap\in\Pi$. Then $\n\oplus\g^{\ap}=\ut$.
If $e\in\n$ is Richardson, then $\dim\g_e=\dim\g-2\dim\n=l+2$
and $G{\cdot}e$ is the subregular nilpotent orbit. 
The parabolic contraction associated with $\p$ is said to be {\it subregular}, too. 
From now on, $\q$ is a subregular contraction of $\g$.
To exclude the case in which $\p=\g$, we assume below that $l\ge 2$.

Recall that the multiset $\{\deg(\cF_1),\dots,\deg (\cF_l)\}$
does not depend on a particular choice of basic invariants in $\eus S(\g)^G$, and if $\g$ is
simple, then there is a unique basic invariant
of maximal degree. (This maximal degree equals the Coxeter number of $\g$.) 
We  assume below that $\cF_l$ has the maximal degree, so that 
$\deg(\cF_i) < \deg(\cF_l)$ for $i< l$. 
The ordering of the previous basic invariants is irrelevant.

\begin{prop}                      \label{l-2}
If\/ $\q$ is a subregular contraction of $\g$, then
\begin{itemize}
\item[\sf (i)] \ $\deg_{\p}(\cF_i^\bullet)=1$ for $i=1,\dots, l-1$ and $\deg_{\p}(\cF_l^\bullet)=2$, 
\item[\sf (ii)] \ the polynomials 
$\cF_1^\bullet,\dots,\cF_l^\bullet$  are algebraically independent
and satisfy the Kostant equality.
\end{itemize}
\end{prop}
\begin{proof} Recall that $\cF_i^\bullet$ is the bi-homogeneous component of $\cF_i$ with
highest $\n_-$-degree.

(i) \ Since $P$ has a dense orbit in $\n$, we have $\eus S(\gt n_-)^P=\bbk$. Therefore the 
$P$-invariant $\cF_i^\bullet$ cannot belong to $\eus S(\gt n_-)\subset \eus S(\q)$ and hence 
$\deg_{\n_-}(\cF_i^\bullet) \le \deg(\cF_i) -1$ for all $i$. 

Consider the bi-homogeneous component of $\cF_l$ with highest $\ut_-$-degree 
(with respect to the decomposition $\g=\be\oplus\ut_-$), denoted by
$\cF_l^\blacktriangle$. It is known that
$\cF_l^\blacktriangle=\displaystyle e_\delta\prod_{i=1}^l f_i^{a_i}$,
where $e_\delta\in\g^\delta$ is a highest root vector, $f_i\in\g^{-\ap_i}$ for $\ap_i\in\Pi$, and 
$\delta=\sum_{i=1}^la_i\ap_i$, see~\cite[Theorem\,3.9 \& Lemma\,4.1]{alafe1}. 
That is, $\cF_l^\blacktriangle$ is a monomial and 
$\deg_{\ut_-}(\cF_l^\blacktriangle) = \deg(\cF_l) -1$. 
Since $\ut_-=\n_-\oplus \g^{-\ap_i}$  for some $i$ and 
all $a_i$ are positive, $\deg_{\n_-}(\cF_l^\blacktriangle) \le \deg(\cF_l) -2$.
This also implies that 
$\deg_{\n_-}(\cF_l^\bullet) \le \deg (\cF_l) -2$.
Therefore, 
\beq       \label{eq:nerav-vo}  \textstyle
   \sum_{i=1}^l\deg_{\n_-}(\cF_i^\bullet) \le \left(\sum_{i=1}^l\deg(\cF_i)\right)-l-1=\dim\n .
\eeq
By Theorem~\ref{t-degrees}, we have  $\sum\deg_{\n_-}(\cF_i^\bullet) \ge\dim\gt n$. Therefore
one actually has the equality, which also means that 
$\deg_{\n_-}(\cF_i^\bullet)= \deg(\cF_i) -1$ for $i\le l-1$ and 
$\deg_{\n_-}(\cF_l^\bullet) = \deg(\cF_l) -2$.

(ii) By Theorem~\ref{t-degrees}, the equality in \eqref{eq:nerav-vo} implies that $\cF_1^\bullet,\dots,\cF_l^\bullet$ are 
algebraically independent and satisfy the Kostant equality.
\end{proof}

In the following lemma, we gather Lie-algebraic properties of the centraliser of a subregular nilpotent
element.

\begin{lm}    \label{lm:subreg-central}
Let $e\in \g$ be a subregular nilpotent element. 
Then 
\begin{itemize}
\item[\sf (i)] \  if\/ $\g$ is not of type $\GR{G}{2}$, then the centre of $\g_e$ is of dimension $l-1$;
if\/ $\g$ is of type $\GR{G}{2}$, then the centre of $\g_e$ is two-dimensional~(see \cite[Theorem\,B]{kurt91});
\item[\sf (ii)] \  if\/ $\g$ is not of type $\GR{G}{2}$, then $\dim [\g_e,\g_e] > 1$.
\item[\sf (iii)] \  if\/ $\g$ is of type $\GR{G}{2}$, then $\g_e$ is the direct sum of\/ $\bbk e$ and 
a three-dimensional Heisenberg Lie algebra $\eus H_3$.
\end{itemize}
\end{lm}
\begin{proof}
(ii) \ Since
$l=\ind \g_e < \dim\g_e=l+2$, $\g_e$ is not abelian, i.e.,  $[\g_e,\g_e]\ne 0$. Assume that 
$\dim [\g_e,\g_e] =1$. Write $\g_e=\z(\g_e)\oplus \ce$, where $\z(\g_e)$ is the centre, and 
$\ce$ is a three-dimensional complement. Since $[\g_e,\g_e]=[\ce,\ce]$ is one-dimensional, 
the space $\ce$ must contain a non-trivial central element. A contradiction!

(iii) \ Let $\Pi=\{\ap,\beta\}$, where $\ap$ is short. One can take $e=e_\beta+e_{3\ap+\beta}$.
Then $\eus H_3=\g^{\ap+\beta}\oplus\g^{2\ap+\beta}\oplus\g^{3\ap+2\beta}$.
\end{proof}
\begin{prop} \label{min-parab-centr}
Let $P\subset G$ be a minimal parabolic subgroup and $e\in\n$ a subregular nilpotent  element. 
Then $\eus S(\g_e)^{P_e}=\eus S(\g_e)^{G_e}$ is freely generated by  
$^{e\!}\cF_1,\dots, ^{e\!}\cF_l$.
\end{prop}
\begin{proof}
By Propositions~\ref{coincidence} and \ref{l-2}(i), we have 
$\deg(^{e\!}\cF_i)=1$ for $i\le l-1$ and $\deg(^{e\!}\cF_l)=2$. In particular, 
$^{e\!}\cF_1,\dots,^{e\!}\cF_{l-1}$ are just elements of $\g_e$. Moreover,
Proposition~\ref{l-2}(ii) implies that $^{e\!}\cF_1,\dots, ^{e\!}\cF_l$ are algebraically independent.
(This also follows from the fact that $\sum_{i=1}^l\deg(^{e\!}\cF_i)=l+1=\frac{1}{2}(\dim\g_e+\ind\g_e)$,
see \cite[Theorem\,2.1]{ppy}.)

Recall that all $^{e\!}\cF_i$ are $G_e$-invariant and hence
$^{e\!}\cF_1,\dots, ^{e\!}\cF_{l-1}$ are linearly independent central elements of $\g_e$.
Then $\z:=\text{span}\{^{e\!}\cF_1,\dots, ^{e\!}\cF_{l-1}\}$ is a central subalgebra of $\g_e$.

\textbullet \quad Suppose that $\g$ is not of type $\GR{G}{2}$. Then $\z=\z(\g_e)$ is the 
centre of $\g_e$.
Consider the coadjoint representation of $G_e^o$ in $\g_e^*$. Since $\dim\g_e=\ind\g_e+2$, the
$G_e^o$-orbits in $\g_e^*$ are of dimension $2$ and $0$. Since $G_e^o$ is connected, the
union of $0$-dimensional orbits is just the subspace $V$ of $\g_e$-fixed points, i.e.,
$V=\{\xi \in \g_e^*\mid x\star \xi=0 \ \ \forall x\in\g_e\}$.
For a linear form $\xi$,  one readily verifies that  $\xi\in V$ if and only if $\xi$ vanishes on $[\g_e,\g_e]$. 
It then follows from Lemma~\ref{lm:subreg-central}(ii) that $\codim V \ge 2$. In other words, the 
set of singular elements in $\g_e^*$ is of codimension $\ge 2$.
Now, combining Theorems~2.1(iii) and~2.2 in \cite{ppy}, we obtain that 
\[
     \eus S(\g_e)^{G_e}=\eus S(\g_e)^{G_e^o}=\bbk[^{e\!}\cF_1,\dots, ^{e\!}\cF_l] .
\]
Since  $G_e^o \subset P_e\subset G_e$, the assertion about $P_e$-invariants follows.

\textbullet \quad Suppose that $\g$ is of type $\GR{G}{2}$.
Then $\g_e$ is the direct sum of $\bbk e$ 
and a Heisenberg Lie algebra $\eus H_3$. Let $(x,y,z)$ be a basis for $\eus H_3$ such that
$[x,y]=z$ is the only non-trivial bracket. Then 
$\z(\g_e)=\bbk e\oplus \bbk z$ and ${\eus S}(\g_e)^{\g_e}=\bbk[e,z]$. 
Note that since $\deg(\cF_1)=2$, we have $^{e\!}\cF_1=e$.
The component group $G_e/G_e^o$ is the symmetric group $\Sigma_3$ and it acts non-trivially on 
$\z(\g_e)$. This can be verified directly, using the element $e$ indicated in the proof of 
Lemma~\ref{lm:subreg-central}(iii).
Since $e$ is a $G_e$-fixed vector, the line $\bbk z\subset \z(\g_e)$ affords  the 
unique non-trivial one-dimensional representation of $\Sigma_3$. Consequently, 
${\eus S}(\g_e)^{G_e}=\bbk[e,z^2]$,
and because $^{e\!}\cF_1,\,^{e\!}\cF_2$ are algebraically independent and
$\deg(^{e\!}\cF_2)=2$, we must have $^{e\!}\cF_2=z^2+ce^2$ for some $c\in\bbk$.
Hence ${\eus S}(\g_e)^{G_e}=\bbk[\,^{e\!}\cF_1,\,^{e\!}\cF_2]$.
\\
There are two minimal 
parabolic subalgebras in $\g$ of type $\GR{G}{2}$. For both of them, $P_e$ is not connected and 
contains an element multiplying  $z$ by $-1$. This again can be verified via direct elementary calculations. 
(Cf. also Remark~\ref{rmk:parabolic-G2} below). Hence 
${\eus S}(\g_e)^{P_e}=\bbk[e,z^2]$ in both cases,  and we are done.
\end{proof}

\begin{rmk}     \label{rmk:parabolic-G2}
For $\ap_i\in\Pi$, let $P_i$ denote the corresponding minimal parabolic in $G$ 
and let $e$ be a subregular element in $\p_i^{nil}$. It was proved in \cite[Prop.\,4.2]{br93} that 
$(P_i)_e=G_e$ if and only if $\ap_i$ is short
(in the simply-laced case, all roots are assumed to be short). Moreover, using the explicit 
description of the Springer fibre of $e$ as a {\it Dynkin curve} \cite[p.147-148]{steinb}, one can
show that if $\ap_i$ is long, then $\#(G_e/(P_i)_e)=\|\ap_i\|^2/\|\ap_{short}\|^2$.
In the $\GR{G}{2}$-case, with 
$\ap_1=\ap$ and $\ap_2=\beta$, we obtain
$(P_1)_e=G_e$ and $\#(G_e/(P_2)_e)=3$. This means that
$(P_2)_e/(P_2)_e^o$ contains an element of order 2 of $\Sigma_3=G_e/G_e^o$, 
which multiplies $z\in \z(\g_e)$ by $-1$.
\end{rmk}

\begin{rmk}
There are other ways to prove Proposition~\ref{min-parab-centr} if $\g$ is not of type $\GR{G}{2}$. Using 
Lemma~\ref{lm:subreg-central} and information on $\{\deg(^{e\!}\cF_i)\}$, one can prove that
$^{e\!}\cF_1,\dots, ^{e\!}\cF_l$ satisfy the hypotheses of Lemma~\ref{lm:igusa} with
$A=G_e^o$, which implies that the functions $^{e\!}\cF_1,\dots, ^{e\!}\cF_l$ freely generate the 
algebra $\eus S(\g_e)^{G_e^o}$ and hence $\eus S(\g_e)^{G_e^o}=\eus S(\g_e)^{G_e}$.
There is also a way to describe $^{e\!}\cF_l$ almost explicitly. The intersection of $e+\g_f$ with the nullcone in $\g$ is isomorphic to a hypersurface in a $3$-dimensional affine 
space with a unique singular point, a Klenian singularity \cite{slodowy}. 
Modulo the ideal $(^{e\!}\cF_1,\ldots,^{e\!}\cF_{l-1})\vartriangleleft \eus S(\g_e)$,
the polynomial $^{e\!}\cF_l$ is the degree $2$ part of the well-known equation defining 
that hypersurface. This statement can be deduced from \cite[Section 7]{Sasha-slice}. 
\end{rmk}

\begin{thm}    \label{thm:min-parab-free-alg}
Let\/ $\q$ be a subregular contraction of\/ $\g$ 
and $\cF_1,\ldots,\cF_l$ the basic invariants in $\eus S(\g)^G$.
Then $\cF_1^\bullet,\dots, \cF_l^\bullet$ \ freely generate $\eus S(\q)^{Q}$
and satisfy the Kostant equality in $\q$.
\end{thm}
\begin{proof}
This readily follows from Proposition~\ref{min-parab-centr} and Theorem~\ref{thm:main3}.
\end{proof}

\end{document}